\pdfoutput=1
\documentclass[a4paper,UKenglish,cleveref, autoref, thm-restate]{lipics-v2021}

\hideLIPIcs  %uncomment to remove references to LIPIcs series (logo, DOI, ...), e.g. when preparing a pre-final version to be uploaded to arXiv or another public repository

\graphicspath{{./images/}}
\title{Local Limit of Random Regular Bipartite Planar Maps} 

\author{Nicolas Tokka}{MODAL'X, Université Paris-Nanterre, France \and IRIF, Université Paris-Cité, France \and \url{https://www.irif.fr/~tokka/} }{nicolas.tokka@irif.fr}{https://orcid.org/0009-0005-2491-8072}{}
\authorrunning{N. Tokka}

\Copyright{Nicolas Tokka} 



\category{} %optional, e.g. invited paper

\relatedversion{} %optional, e.g. full version hosted on arXiv, HAL, or other respository/website
\funding{ANR IsOMa "ANR-21-CE48-0007".}

\acknowledgements{We sincerely thank Marie Albenque and Laurent Ménard for their thoughtful review and valuable advice throughout this work.}

\nolinenumbers %uncomment to disable line numbering

\EventEditors{John Q. Open and Joan R. Access}
\EventNoEds{2}
\EventLongTitle{42nd Conference on Very Important Topics (CVIT 2016)}
\EventShortTitle{CVIT 2016}
\EventAcronym{CVIT}
\EventYear{2016}
\EventDate{December 24--27, 2016}
\EventLocation{Little Whinging, United Kingdom}
\EventLogo{}
\SeriesVolume{42}
\ArticleNo{23}
\usepackage{amsfonts}
\usepackage{mathrsfs} %Nice script letters (curved)
\usepackage{stmaryrd} %for \Mbr and bracket
\usepackage{mathtools}
\usepackage{dsfont} %for characteristic function
\usepackage{bm} %for bold grec math symbols
\usepackage{nicematrix}

\newcommand{\R}{\mathbb{R}}

\providecommand{\Z}{}

\renewcommand{\Z}{\mathbb{Z}}

\renewcommand{\P}{\mathbb{P}}
\newcommand{\cK}{\mathcal{K}}

\newcommand{\cT}{\mathcal{T}}

\newcommand{\sW}{\mathsf{W}}
\newcommand{\sw}{\bm{\mathsf{w}}}

\newcommand{\sP}{\mathsf{P}}

\newcommand{\s}{\mathsf{s}}

\newcommand{\rV}{\mathrm{V}} 
\newcommand{\rE}{\mathrm{E}} 
 
\newcommand{\rt}{\mathrm{t}} 
\newcommand{\rf}{\mathrm{f}}
\renewcommand{\rm}{\mathrm{m}}
\newcommand{\rS}{\mathrm{S}}
\newcommand{\cS}{\mathcal{S}}
\newcommand{\cC}{\mathcal{C}}

\renewcommand{\c}{\mathrm{c}}
\newcommand{\cW}{\mathcal{W}}
\newcommand{\cB}{\mathcal{B}}
\newcommand{\rW}{\mathrm{W}}
\newcommand{\rB}{\mathrm{B}}

\renewcommand{\S}{\mathcal{S}^d}

\newcommand{\brf}{\bm{\mathrm{f}}} 
\newcommand{\M}{\mathcal{M}^d}

\newcommand{\Mc}{\overline{\mathcal{M}}^d}
\newcommand{\Mf}{\mathcal{M}^{d,\brf}}
\newcommand{\Mfn}{\mathcal{M}^{d,\brf}_n}

\newcommand{\T}{\mathcal{T}^d}
\newcommand{\Tn}{\mathcal{T}^d_n}
\newcommand{\Tc}{\overline{\mathcal{T}}^d}

\usepackage{tikz}
\newcommand{\cstem}{% Closing stem
  \tikz[baseline=0.3ex, line cap=round, rotate=90]{
    \draw[line width=0.5pt] (0,0) -- (0.45em,0);
    \draw[line width=0.5pt] (0.45em,0) -- ++(0.8ex,0.4ex);
    \draw[line width=0.5pt] (0.45em,0) -- ++(0.8ex,-0.4ex);
    \draw[line width=0.5pt] (0.45em+0.8ex,0.4ex) -- ++(0ex,-0.8ex);
    \fill[line width=0.5pt] (0.45em,0) -- ++(0.8ex,0.4ex) -- ++(0,-0.8ex) -- cycle;}}
\newcommand{\ostem}{% Opening stem
  \tikz[baseline=0.3ex, line cap=round, rotate=90]{
    \draw[line width=0.5pt] (0,0) -- (0.45em,0);
    \draw[line width=0.5pt] (0.45em,0.4ex) -- ++(0.8ex,-0.4ex);
    \draw[line width=0.5pt] (0.45em+0.8ex,0) -- ++(-0.8ex,-0.4ex);
    \draw[line width=0.5pt] (0.45em,-0.4ex) -- ++(0,0.8ex);}}

\newcommand{\1}{\mathds{1}}
 
\newcommand{\dloc}{\mathrm{d}_\mathrm{loc}}

\begin{document}
\maketitle

\begin{abstract}
We prove the existence of the local limit of uniform random 
$d$-regular bipartite planar maps, for every $d\geq 3$, as the number of vertices tends to infinity. 
The proof relies on a bijection between maps and so-called blossoming trees established in a previous work. After proving local convergence of the associated decorated trees, we extend the bijection to infinite trees and transfer the convergence to planar maps. The limiting object is almost surely one-ended and recurrent for the simple random walk.
\end{abstract}

\section{Introduction}
In the last two decades, the study of local limits of large random planar maps has attracted significant interest. The subject was initiated by the seminal work of Angel and Schramm~\cite{AngelSchramm03}, who proved in 2003 that uniform planar triangulations with $n$ faces converge in distribution, for the local topology, to a probability measure supported on infinite triangulations. The limiting object is now known as the Uniform Infinite Planar Triangulation (UIPT). This result laid the foundations for the investigation of local limits of large random maps.
\medskip

Shortly thereafter, Krikun~\cite{Krikun06} established the analogous convergence of uniform quadrangulations toward the Uniform Infinite Planar Quadrangulation (UIPQ). Independently, Chassaing and Durhuus constructed another model of random infinite planar quadrangulation. Their approach relies on a powerful bijective perspective, using the Cori–Vauquelin–Schaeffer bijection~\cite{Schaeffer98} between quadrangulations and decorated plane trees. Afterwards, Ménard~\cite{Menard10} proved that both these infinite random quadrangulations have the same law. 

Chassaing and Durhuus’s method can be summarized in two steps. First, prove the local convergence of the associated random trees obtained by applying the Cori–Vauquelin–Schaeffer bijection. Second, extend this bijection to the case of infinite trees to define their infinite random quadrangulation. 

This general strategy has proved robust and has since been implemented in a variety of settings. It has been notably implemented via the generalized bijection known as the \emph{mobile bijection}~\cite{BDG04}. This framework has led to a large number of results on local limits of random maps, e.g.~\cite{CurienMenardMiermont13,MenardNolin14,Stephenson18,Stufler22}.

Another powerful bijective scheme has been used similarly to define families of infinite random maps. There are called \emph{mating-of-trees} bijections. These bijections generalize Mullin's bijection~\cite{Mullin67} and encode some families of decorated maps by two-dimensional walks, see for example~\cite{Chen17,GwynneHoldenSun23,Sheffield16}.

Our work is based on a third bijective paradigm between maps and \emph{blossoming trees} with charge constraints. While this paradigm was introduced twenty years ago~\cite{Schaeffer97,BousquetMelouSchaeffer03,BDG02}, it was not clear how to use it to derive probabilistic convergence results. The present work addresses precisely this gap using a recent version of the blossoming bijection due to Albenque, Ménard, and the author~\cite{AMT25}.
Specifically, we study local limits of uniform $d$-regular bipartite planar maps. Precise definitions are given in Section~\ref{sec1}. Our main result establishes the local convergence of these maps, as their number of vertices tends to infinity, toward a limiting random infinite map that we call the Uniform Infinite $d$-Regular Bipartite Planar Map ($d$-UIRBPM), see Theorem~\ref{theo:LocalConvergencedRegularMaps}.
\smallskip

Previous works also studied local convergence of bipartite maps~\cite{BjrnbergStefansson14,Carrance21,Stephenson18}, but we want to emphasize that, in our work, we study maps with vertex degree $d$, and which are furthermore required to be bipartite. In contrast, the models of maps studied in~\cite{BjrnbergStefansson14,Stephenson18} consist of maps with only control on face degrees, and which are bipartite (i.e. all faces are required to have even degree).  The only work with similar constraints is the convergence of Eulerian planar triangulations by Carrance~\cite{Carrance21}, which relies in particular on the so-called layer decomposition, which is both technical and specific to triangulations. The special case $d=3$ of the present article recovers this result in the dual setting. Moreover, the fact that we can have a simpler proof stems from the fact that, as already mentioned, we use another type of bijections.  In contrast to Carrance's approach, our work is more robust and allows to fully control the vertex degree distribution as we plan to discuss in future work, see Section~\ref{sec4}. Moreover, to our knowledge, this is the first proof of a local limit theorem for planar maps that is based entirely on the blossoming-tree paradigm, rather than on the mobile bijection aforementioned.

We proceed by first studying the local convergence of uniform well-charged trees corresponding to $d$-regular bipartite maps. We then show that these random trees converge in distribution to an explicit infinite plane tree with a unique spine, that is, a single infinite branch starting from the root. Moreover, we identify the limiting distribution as that of a multitype Bienaymé–Galton–Watson tree conditioned to survive. This is achieved in Section~\ref{sec2}.
In a second step, we extend the blossoming bijection to infinite trees and prove its continuity on the support of the limiting tree law. This allows us to transfer the local convergence of trees to the local convergence of maps, and therefore establishing the existence of the $d$-UIRBPM. This is covered in Section~\ref{sec3}.

The nature of $d$-UIRBPM implies directly that it is almost surely recurrent for the simple random walk. This follows from a theorem of Gurel-Gurevich and Nachmias in~\cite{GurelGurevichNachmias13}, which ensures recurrence for local limits of uniformly rooted finite maps whose root vertex degree has an exponential tail. This is automatically satisfied in our regular setting.
\smallskip

As mentioned above, beyond the regular case, the method developed here appears to be flexible and suggests several extensions, which are discussed in Section~\ref{sec4}. We strongly believe that uniform bipartite planar maps with bounded degrees and prescribed degree distributions also converge, for the local topology, to the distribution of an infinite bipartite planar map. In addition, this generalization for bipartite maps with vertex degree $2$ and $d$ could lead to the study of local limits of maps decorated by statistical physics models such as the Ising model or the hard particle model. We plan to tackle these questions in future work and give some indications of how to address them in Section~\ref{sec4}.

\section{Definitions and main result}\label{sec1}

\subsection{Maps}A \emph{planar map} is a proper embedding of a connected planar graph on the 2-dimensional sphere $\mathbb{S}^2$, considered up to orientation-preserving homeomorphisms. 

The \emph{edges} and \emph{vertices} of a map are the natural counterparts of the edges and vertices of the underlying graph. The \emph{faces} of a map $\rm$ are the connected components of the complement of its embedded graph. The sets of vertices, edges, and faces are denoted by $\mathrm{V}(\rm)$, $\mathrm{E}(\rm)$ and $\mathrm{F}(\rm)$, respectively. Note that loops and multiple edges are allowed. 

The \emph{corners} of a map are ordered pairs of consecutive half-edges around a vertex, where a \emph{half-edge} is one of the two components obtained by splitting an edge at its midpoint. The \emph{degree} of a vertex or a face is defined as the number of its incident corners. A map is \emph{$d$-regular} if all its vertices have degree $d$.

\smallskip

Maps are assumed to be \emph{rooted}, meaning that one of their corners is designated as the origin. This corner is called the root corner and is indicated by a double arrow in figures. The vertex and the face incident to this corner are called the \emph{root vertex} and the \emph{root face}, respectively.

A \emph{plane map} $(\rm,\rf)$ is a pair consisting of a map $\rm$ together with a marked face $\rf\in\mathrm{F}(\rm)$. A plane map has a canonical embedding in the plane via stereographic projection, by choosing the \emph{outer infinite face} as the marked face. In particular, the outer face is not necessarily its root face. Finally, a \emph{plane tree} is a plane map with a single face.

\paragraph*{Bipartite maps}
A \emph{bipartite} map is a map that admits a proper $2$-coloring of its vertices. Note that a bipartite map has exactly two proper 2-colorings, which differ only by swapping the color of all vertices. 

In this work, \emph{each bipartite map will be endowed with one of these two colorings}, say in black and white. We denote the black and white vertices of a bipartite map $\rm$ by $\rV_{\!\bullet}(\rm)$ and $\rV_{\!\circ}(\rm)$, respectively. Finally, a rooted bipartite map is said to be \emph{black-rooted} or \emph{white-rooted} depending on the color of its root vertex.

\paragraph*{Regular bipartite maps}
For any $d \geq 3$, we denote by $\M$ the set of all finite, black-rooted, $d$-regular bipartite planar maps. 

\begin{remark}
	Note that every $d$-regular bipartite planar map has the same number of black and white vertices. Thus, the total number of vertices is even. Conversely, for any $n\geq 1$, there exists such a map with exactly $2n$ vertices.
\end{remark}

For any $n \ge 1$, we denote by $\mathcal{M}^d_n$ the subset of $\M$ consisting of maps with $n$ black vertices. Similarly, the corresponding sets of such maps with an additional marked face are denoted by $\Mf$ and $\Mfn$, respectively.

Denote by $M^d_{n}$ the number of maps in $\mathcal{M}^d_n$, for all $n\geq 1$, and let $M^d(z)$ be its associated generating series, i.e.:
\begin{equation*}
	M^d(z) \coloneqq \sum\limits_{\rm\in\M}z^{\vert\rV_{\!\bullet}(\rm)\vert} = \sum\limits_{n\geq 1}M^d_{n}z^n.
\end{equation*}
Similarly, define ${M}^{d,\brf}_n$ and  ${M}^{d,\brf}(z)$ for maps in $\Mf$. A straightforward application of Euler's formula shows that:
\begin{equation}\label{eq:NbPlaneMaps}
	{M}^{d,\brf}_n = \left(2+(d-2)n\right)M^d_n.
\end{equation}

\subsection{Local convergence of random maps}
We first define a distance on $\M$. For $\rm\in\M$ and $R\geq 0$, let $B_R(\rm)$ denote the submap of $\rm$ consisting of all vertices at graph distance at most $R$ from the root vertex, together with all edges that have at least one extremity at distance less or equal to $R-1$ from the root vertex. Moreover, $B_R(\rm)$ inherits the root corner from $\rm$. For $\rm, \rm'\in\M$, we set:
\begin{equation}\label{eq:DefLocDist}
	\dloc\left(\rm,\rm' \right)\coloneqq {\left(1+\mathrm{sup}\left\{R\geq 0 : B_R(\rm)=B_R(\rm') \right\}\right)}^{-1}.
\end{equation}
The topology on $\M$ induced by $\dloc$ is known as the \emph{local topology}. The closure $(\Mc,\dloc)$ of the metric space $(\M,\dloc)$ is a Polish space, and the elements of $\Mc\setminus\M$ are referred to as  \emph{infinite $d$-regular bipartite rooted planar maps}. Finally, such a map $\rm$ is said to be \emph{one-ended} if, for every finite submap $\rm'\subset\rm$, exactly one connected component of $\rm\setminus \rm'$ contains infinitely many faces.
\medskip

Our main result is the following. 

\begin{theorem}\label{theo:LocalConvergencedRegularMaps}
For any $n\geq 1$, denote by $\P^d_n$ the uniform probability measure on $\mathcal{M}^d_{n}$.

Then, the sequence $\left(\P^d_n\right)_n$  converges weakly, for the  local topology, to a limiting probability measure $\P^d_\infty$ supported on the set of one-ended infinite $d$-regular bipartite planar maps.
\end{theorem}

The proof relies on a bijection between plane bipartite maps and a certain class of decorated trees, which we define in the next section. We first prove local convergence of this class of trees. Next, we show that the bijection extends naturally to the infinite setting, which allows us to conclude the proof.

\section{Local convergence of well-charged trees}\label{sec2}

In this section, we define a family of decorated plane trees, first introduced in~\cite{BousquetMelouSchaeffer03} and called \emph{well-charged trees}, which are in bijection with bipartite plane maps. We then analyze the local convergence of the uniform distribution on these trees as their size grows to infinity. 

\subsection{Bijection between bipartite plane maps and well-charged trees}

\paragraph*{Definition of well-charged trees}
In this work, a \emph{blossoming tree} is a \emph{bipartite} plane tree in which each corner may carry some half-edges. Here, bipartite means that the tree is endowed with one of its two proper $2$-colorings, in accordance with the definitions given in Section~\ref{sec1}. A half-edge attached to a black vertex is oriented outward and is referred to as an \emph{opening stem}, and a half-edge incident to a white vertex is oriented inward and is called a \emph{closing stem}. A rooted blossoming tree $\rt$ is called \emph{white} or \emph{black} according to the color of its root vertex, denoted by $\varnothing_\rt$. 

The \emph{offspring} of a vertex is the ordered collection of its incident half-edges, excluding the one that leads to its parent. The \emph{height} of a vertex is its distance from the root vertex, and the \emph{height} of a stem is the height of its unique incident vertex.

The \emph{charge} of a blossoming tree $\rt$ is the difference between the number of its closing stems and opening stems, and is denoted by $\c(\rt)$. If $\rt$ is rooted, then the \emph{charge} of a vertex  $v\in\rV(\rt)$ is defined as the charge of the subtree of $\rt$ rooted at $v$ and is denoted by $\c_\rt(v)$, or simply $\c(v)$ when the context is clear.

Finally, a rooted blossoming tree $\rt$ is \emph{well-charged} if it meets the following charge conditions: 
\begin{equation}\label{eq:ChargedConditions}
\left({C}\right):~\begin{cases}
\c_\rt(v)\leq 1, ~\forall v\in\rV_{\!\bullet}(\rt),\\
\c_\rt(v)\geq 0, ~\forall v\in\rV_{\!\circ}(\rt).
\end{cases}
\end{equation}

For $n\geq 1$, we denote by $\T_n$ the set of black-rooted well-charged $d$-regular trees of charge $0$ with $n$ black vertices, and we write $\T\coloneqq\sqcup_{n\geq 1}\T_n$.

\paragraph*{Closure of well-charged trees}
Following~\cite{AMT25,BousquetMelouSchaeffer03}, we now describe the \emph{closure} of a blossoming tree of charge $0$. Let $\rt$ be a blossoming tree of charge $0$. Consider the ordered, cyclic sequence of its opening and closing stems obtained by performing a clockwise walk around $\rt$. 

\begin{figure}[t!]
  	\centering
  	\includegraphics[width=0.9\linewidth ,page=1]{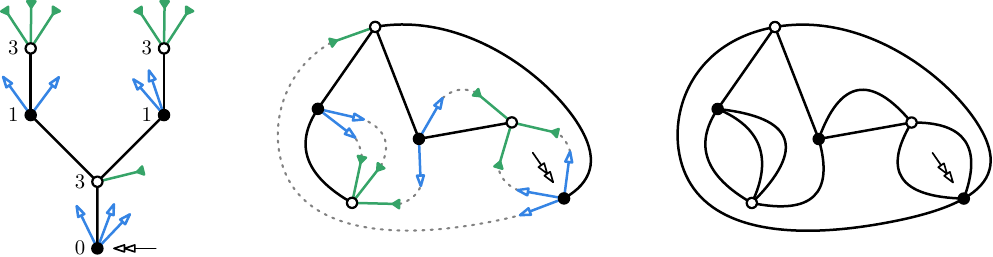}
 	\caption{A tree $\rt$ in $\mathcal{T}^4_3$ with its charge function (left), the same tree drawn differently with the matching of its stems (middle), and its closure (right).}\label{fig:Closure}
\end{figure}

The sequence induces a perfect matching between the opening and closing  stems as in a parenthesis word. The closure is then obtained by merging each matched pair of stems into a single edge. Each edge is drawn in the plane so that, when oriented from the opening stem to the closing stem, the infinite face lies to its left. 

We can prove that the resulting object is a plane map and that it does not depend of the order in which the edges are added.

\begin{proposition}[Theorem 2.5, \cite{AMT25}, see also~\cite{BousquetMelouSchaeffer03}]\label{theo:BijMapsTrees}
For any $d\geq 3$ and $n\geq 1$, the \emph{closure} operation $\Phi$ is a one-to-one correspondence between $\T_n$ and $\Mf_n$. In particular, 
\begin{equation*}
\vert\T_n\vert = M^{d,\brf}_n.
\end{equation*}
\end{proposition}

\subsection{Convergence of random well-charged trees}
Let $(\Tc,\dloc)$ denote the Polish space obtained as the closure of the metric space $\left(\T,\dloc\right)$ under the local topology, which is defined analogously to the case of maps. The elements of $\Tc\setminus\T$ are referred to as \emph{infinite $d$-regular blossoming trees}. A \emph{spine} of a tree $\rt\in\Tc$ is an infinite path of $\rt$ starting at its root vertex. 

We now extend the notion of charge to infinite trees in $\Tc$. Let $T$ be such a tree. A \emph{charge function} on $T$ is a function $\c:\rV(T)\rightarrow\Z$ satisfying the following conditions:
\begin{equation*}
	\c(\varnothing_T)=0 \quad\text{and}\quad \forall v\in\rV(T), ~ \c(v)=n_{ \cstem}(v) - n_{ \ostem}(v) + \sum_{1\leq j\leq k}{\c(u_j)},
\end{equation*}
where $n_{ \cstem}(v)$ and $n_{ \ostem}(v)$ denote the numbers of closing and opening stems incident to $v$, respectively, and $u_1,\ldots,u_k$ denote the vertices incident to $v$, different from its parent. We obtain the following by induction.

\begin{claim}
If $T$ has a single spine, then it admits a unique charge function. In that case, we say that $T$ is \emph{well-charged} if its charge function satisfies the conditions $({C})$ given in~\eqref{eq:ChargedConditions}
\end{claim}

The main result of this section is the following. 
\begin{theorem}\label{theo:WeakConvergencedRegularTrees}
For any $n\geq 1$, denote by $\sP_n^d$ the uniform probability distribution on $\Tn$.

Then, the sequence $\left(\sP^d_n\right)_n$  converges weakly, with respect to the  local topology, to a limiting probability measure $\sP^d_\infty$ supported on the set of infinite $d$-regular well-charged blossoming trees with a single spine.
\end{theorem}

\subsection{Proof of Theorem~\ref{theo:WeakConvergencedRegularTrees}}\label{subsec:ProofLocalConvTree}
Let $\bm{T}_n$ be a random tree with distribution $\sP^d_n$. Because (almost surely) all the vertices of $\bm{T}_n$ have degree $d$, the sequence $(\sP^d_n)_n$ is \emph{tight} for the local topology. Therefore, to prove Theorem~\ref{theo:WeakConvergencedRegularTrees}, excluding the single spine property, it is enough to show that for any $k\geq 0$ and $\rt\in\T$,
\begin{equation}\label{eq:AsymptoticToCompute}
	\lim_{n\to\infty}\P\big( B_k(\bm{T}_n) = B_k(\rt)\big) \text{ exists}.
\end{equation}
To this end, we first study in greater detail the structure of the finite well-charged trees.

\paragraph*{Structure of well-charged trees and generating functions}

We begin by investigating the possible charge values and offspring for trees in~$\T$. A direct consequence of the degree constraints yields the following result. 
\begin{lemma}\label{Cla:ChargeOffspringRegular}
Let $\rt\in\T$. Then, for any $v\in\rV(\rt)$:
\begin{itemize}
	\item If $v\in\rV_{\!\circ}(\rt)$, then $\c(v)=d-1$. Moreover, its ordered offspring consists of $d-1$ elements, each of which is either a closing stem or a black vertex.
	\item If $v\in\rV_{\!\bullet}(\rt)$, then $\c(v)=0$ if $v$ is the root vertex and $\c(v)=1$ otherwise. Moreover, its ordered offspring  consists of a single white vertex together with opening stems, exactly $d-1$ if $v$ is the root vertex and $d-2$ otherwise.
\end{itemize}
Moreover, all the possibilities can occur. 
\end{lemma}

Next, for enumerative purposes, we study the \emph{recursive decomposition} of well-charged trees. A \emph{planted} blossoming tree is a non-rooted blossoming tree with an additional distinguished dangling half-edge, which is neither an opening nor a closing stem. The  unique vertex incident to this distinguished half-edge is called the \emph{root vertex}. Such a tree is called \emph{black} or \emph{white}, depending on the color of its root vertex. Finally, a planted blossoming tree is \emph{well-charged} if it satisfies the same charge conditions $({C})$ as before, see~\eqref{eq:ChargedConditions}.

\begin{figure}[t!]
  	\centering
  	\includegraphics[width=0.6\linewidth ,page=3]{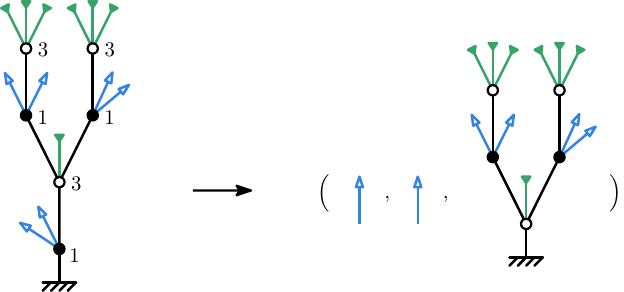}
 	\caption{A tree in $\cB$ with its charge function, and its root-decomposition. Here $d=4$.}\label{fig:RootDecomp}
\end{figure}

The \emph{root-decomposition} of a rooted or planted well-charged tree $\rt$ consists of the finite sequence, ordered according to the planar embedding, of stems and planted well-charged trees incident to the root vertex.  
\medskip

Let $\cB$ be the set of all finite planted well-charged \emph{black} trees with charge $1$, and let $\cW$ be the set of all finite planted well-charged \emph{white} trees with charge ${d-1}$. 
Denote by $\rB(z)$ and $\rW(z)$ their associated generating series:
\begin{equation*}
	\rB(z) = \sum\limits_{\rt\in\cB}z^{\vert\rV_{\!\bullet}(\rt)\vert} \quad\text{and}\quad \rW(z) = \sum\limits_{\rt\in\cW}z^{\vert\rV_{\!\bullet}(\rt)\vert}.
\end{equation*}
Then, we have the following result. 

\begin{lemma} 
The generating functions satisfy the following relations:
\begin{equation*}
	\rW(z)= (1+ \rB(z))^{d-1} \text{,}\quad \rB(z) = z(d-1)\rW(z) \text{,}\quad {M}^{d,\brf}(z) = zd\,\rW(z),
\end{equation*} 
and $\rB(z)$ is characterized by the Lagrangian equation:
\begin{equation}\label{LagrangianEqB1}
	\rB(z) = z(d-1)(1+ \rB(z))^{d-1}.
\end{equation} 
Moreover, these generating functions share the common radius of convergence
\begin{equation*}
	\rho_d=(d-2)^{(d-2)}/(d-1)^d
\end{equation*}
\end{lemma}
\begin{proof}
The first two relations follow from Lemma~\ref{Cla:ChargeOffspringRegular}. The statement concerning the offspring of the root vertex of a tree $\rt\in\T$ of this lemma, combined with Proposition~\ref{theo:BijMapsTrees} yields the third relation.

The Lagrangian equation characterizing $\rB(z)$ follows by direct substitution, and allows one to determine its radius of convergence, see~\cite[Prop IV.5]{FS09}. Finally, we can derive explicit expressions of $\rW(z)$ and ${M}^{d,\brf}(z)$ in terms of $\rB(z)$, from which we conclude that they share the same radius of convergence.
\end{proof}

\paragraph*{Explicit computation of~\eqref{eq:AsymptoticToCompute}}
In order to prove the existence of~\eqref{eq:AsymptoticToCompute}, as $n$ tends to infinity, and compute its value, we introduce the following notation. We denote by $m_k(\rt)$ the number of vertices of $\rt$ at height $k$, and we set $n_k(\rt)=\vert\rV_{\!\bullet}(B_k(\rt))\vert$. To begin, assume that $k$ is odd. Thus, the cut of $\rt$ at height $k$ leads to $B_k(\rt)$ together with an ordered collection of stems and $m_k(\rt)$ trees in $\cW$. This yields the following expression:
\begin{equation*}
	\P\big( B_k(\bm{T}_n) = B_k(\rt)\big) = \frac{1}{{M}^{d,\brf}_n}\cdot [z^{n-n_k(\rt)}]\rW(z)^{m_k(\rt)} = \frac{[z^{n-n_k(\rt)}]\rW(z)^{m_k(\rt)} }{d\,[z^{n-1}]\rW(z)},
\end{equation*}
where $[z^n]$ extracts the coefficient of $z^n$ in a formal power series. Both numerator and denominator can be computed explicitly by applying the Lagrange inversion formula~\cite[Theorem A.2]{FS09} on $\rB(z)$ via~\eqref{LagrangianEqB1}. The previous ratio is then equal to
\begin{equation*}
 	\frac{(d-2)n+1}{n-n_k(\rt)}\,\frac{d\,m_k(\rt)\,(d-1)^{n-n_k(\rt)}}{d\,(d-1)^{n-1}}\,\binom{(d-1)(n+m_k(\rt)-n_k(\rt))-1}{n-n_k(\rt)-1}\binom{(d-1)n}{n}^{-1}.
\end{equation*}
Thus, we obtain
\begin{equation}\label{eq:ProbaToInterpret}
	\P\big(B_k(\bm{T}_n) = B_k(\rt)\big) \underset{n\to\infty}{\longrightarrow} m_k(\rt)\,{\rho_d}^{n_k(\rt)}\,\frac{(d-1)(d-2)}{d}  \left(\frac{d-1}{d-2}\right)^{(d-1)\, m_k(\rt)}.
\end{equation}
The case where $k$ is even is similar and is left to the reader.

This completes the proof of local convergence of $(\sP^d_n)_n$ to a limiting probability measure $\sP^d_\infty$ supported on $\Tc$.
We can interpret $\sP^d_\infty$, via~\eqref{eq:ProbaToInterpret}, as the law of an explicit  well-chosen multitype Bienaymé-Galton-Watson distribution. This yields the following result, which concludes the proof of Theorem~\ref{theo:WeakConvergencedRegularTrees}. 
\begin{proposition}\label{prop:PropertiesOnInfiniteRandomTree}
Let $\bm{T_\infty}$ be a random tree sampled according to $\sP^d_\infty$. Then $\bm{T_\infty}$  almost surely has a unique spine and is well-charged. 

Moreover, each vertex has an offspring as described in Lemma~\ref{Cla:ChargeOffspringRegular}, and, conditioned on the offspring, the order of children and stems is uniform. 
\end{proposition}
This proposition provides additional properties of $\bm{T_\infty}$ that are useful later. Its proof is deferred to Appendix~\ref{appA}, as it relies on the theory of multitype Bienaymé-Galton-Watson trees, which is not used elsewhere in this article.

\section{Local convergence of regular bipartite maps}\label{sec3}

In this section, we prove Theorem~\ref{theo:LocalConvergencedRegularMaps}. To this end, we first extend the closure operation $\Phi$ to $\S\coloneqq\mathrm{Supp}(\sP^d_\infty)$. We then show that this extension defines a continuous mapping, with respect to the local topology. Finally, we conclude by applying $\Phi$ to trees distributed according to $\sP^d_\infty$.

\subsection{Closure of infinite blossoming trees with one spine}
First, it is well known that infinite trees with a single spine have a unique embedding in the plane with a unique condensation point at $\infty$,  up to homeomorphism, see for example~\cite{ChassaingDurhuus06}. 

To extend the closure to $\cS^d$, we proceed as in the case of finite blossoming trees. We first define the matching of stems, and then we merge each matched pair into an edge. Figure~\ref{fig:InfiniteClosureRegular} illustrates the following constructions.

\paragraph*{Matching of stems} Let $T\in\cS^d$. We define a bi-infinite word $\bm{\omega}_T\coloneq\ldots \omega_{-2}\omega_{-1} \omega_{0} \omega_{1} \ldots$, with letters in $\{\cstem, \ostem\}$, as follows. Starting from the root, we perform a walk to the right along the tree and record the stems in the order in which they are encountered: for the $l$-th stem, we set $\omega_{l-1} = \cstem$ if it is a closing stem and $\omega_{l-1} = \ostem$ if it is an opening stem. We then proceed analogously to the left of the root, defining $\omega_{-l}$ for the $l$-th stem encountered. The resulting word $\bm{\omega}_T$ is called the \emph{contour word of stems} of $T$. 

The \emph{index} of the $l$-th stem encountered is defined to be  $l-1$ if the stem is encountered to the right of the root and  $-l$ if it is encountered to the left.

\begin{figure}[t!]
  	\centering
  	\includegraphics[width=0.35\linewidth ,page=2]{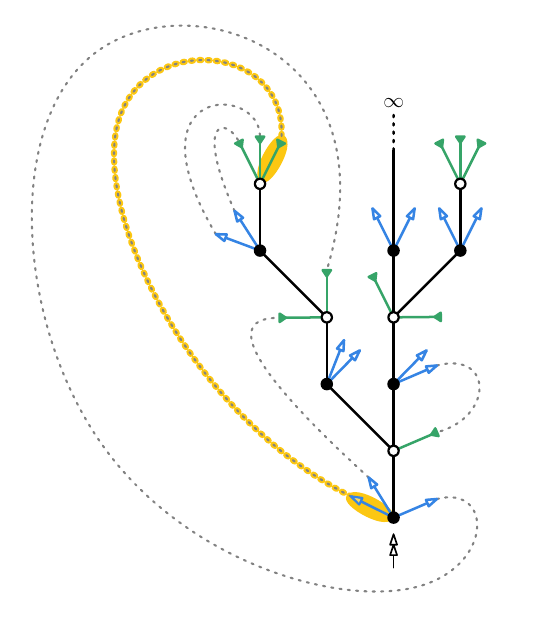}\qquad
  	\includegraphics[width=0.55\linewidth ,page=3]{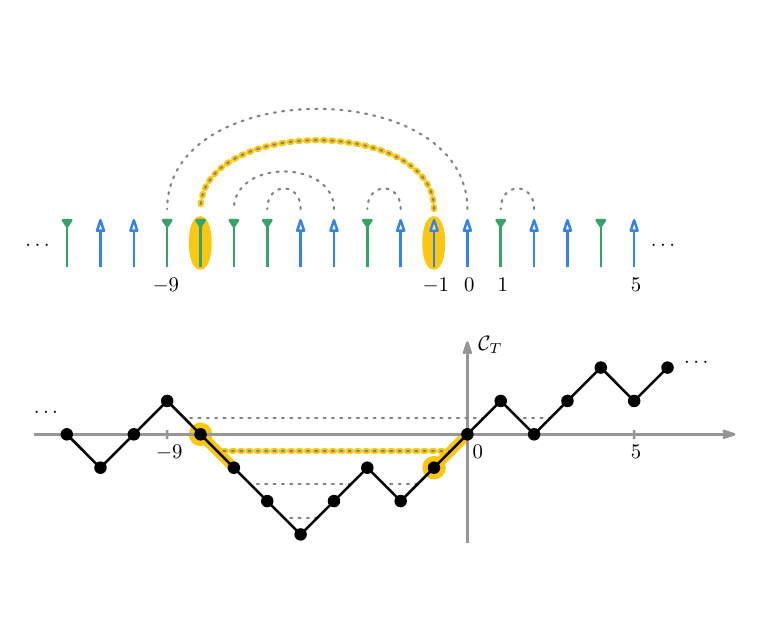}
 	\caption{A tree in $\cS^d$ (left), its contour word of stems, and its associated contour walk of stems (right). In each figure, the same partial matching of stems is represented in dotted lines, and a pair of matched stems is highlighted.}\label{fig:InfiniteClosureRegular}
\end{figure}
\medskip

Next, we define the \emph{contour walk of stems} $\cC_T\!:\Z\rightarrow\Z$, as the mapping defined by $\cC_T(0) = 0$, and for all $k\in\Z$:
\begin{equation*}
	\cC_T(k+1)-\cC_T(k) \coloneqq 
	\begin{cases}
		-1 & \text{if } \omega_k=\cstem,\\
		+1 & \text{if }  \omega_k=\ostem.	
	\end{cases}		
\end{equation*}
In other words, for $k\geq 0$, $\cC_T(k)$ is equal to the number of opening stems minus the number of closing stems in $\omega_{0}\ldots\omega_{k-1}$, whereas for $k< 0$, it is equal to the number of closing stems minus the number of opening stems in $\omega_{-k}\ldots\omega_{-1}$.
\medskip

Finally, we extend the notion of matching of stems to the case of infinite trees. The following formalism makes precise the construction used in the finite case. The \emph{matching of stems} of $T$ is a mapping $\sigma_T\!:\Z\rightarrow\Z\sqcup\{-\infty,+\infty\}$, with no fixed point, defined as follows:
\begin{equation}\label{eq:DefSigma}
	\sigma(k) \coloneqq
		\begin{cases}
			\sup\left\{j<k :~ \omega_j=\cstem,~ \cC_T(j)=\cC_T(k)-1\right\} & \text{ if } \omega_k=\ostem,\\
			\inf\left\{j>k :~ \omega_j=\ostem,~ \cC_T(j)=\cC_T(k)+1\right\} & \text{ if } \omega_k=\cstem.
		\end{cases}
\end{equation}
where by convention $\inf\varnothing=+\infty$ and  $\sup\varnothing=-\infty$. It follows immediately that if $\sigma_T(k)=j\notin\{-\infty,+\infty\}$, then $\sigma_T(j)=k$ and $\{\omega_k,\omega_j\}=\{\cstem,\ostem\}$.
\medskip

\begin{proposition}\label{prop:PerfectMatching}
The mapping $\sigma_T$  defines a perfect matching between the closing and opening stems of $T$, that is, $\sigma_T$ is a bijection from $\Z$ to $\Z$. 
\end{proposition}

This follows directly from the lemma below, whose proof is deferred to Section~\ref{sec:ProofLemma}.

\begin{lemma}\label{pro:ClosureIsWellDefined}
For every $T\in\S$, one has
\begin{equation*}
	 \sup_{k\geq 0}\cC_{T}(k)= \sup_{k< 0}\cC_{T}(k)=+\infty,
\end{equation*}
\end{lemma}

\paragraph*{Closure operation} 
We extend the closure operation to trees in $\cS^d$. Let $T\in\cS^d$. We define its closure $\Phi(T)$ as follows: starting from the root, we perform a walk along the tree to the right. Whenever an opening stem is encountered, say the stem with index $l$, we merge it with the stem with index $\sigma_T(l)$ into a single edge. The edge is drawn so that it does not cross any edge, vertex, or stem, and such that if it were oriented from the opening stem to the closing stem, then the outer infinite face would lie on its left. 

After this merging step, at most two faces have incident stems, one of which is the infinite face. On the finite face, the number of opening and closing stems coincide and is finite. We perform the same merging procedure there. 

Once these mergings have been completed, we continue the walk around the tree until the next opening stem, and we repeat the procedure. Figure~\ref{fig:InfiniteClosureRegular} illustrates the  first steps of the closure procedure. 

By construction, this leaves us with an infinite map $M=\Phi(T)$, meaning that it admits an embedding in the plane such that every bounded subset of the plane intersects only finitely many edges and vertices. This can be proved similarly to the argument described in~\cite{Stephenson18}. In addition, $M$ is one-ended since $T$ has a single spine.

Moreover, we can identify the corners of $T$ with those of $M$, and we canonically root $M$ at the root corner of $T$. Note that, contrary to the finite case, $M$ has no marked face. 

Finally, note that no unmatched stems remain in $M$ by Proposition~\ref{prop:PerfectMatching}. Thus, the closure of $M$ belongs to $\Mc$.

\subsection{Continuity of the closure on the support of \texorpdfstring{$\sP^d_\infty$}{Pdinf}}

The purpose of this section is to prove the following Theorem:

\begin{theorem}\label{theo:ContinuityPhi}
	The extended closure operation $\Phi$ defines a continuous mapping from $\S$ to $\Mc$, with respect to the local topology.
\end{theorem}

Before proving this result, we first introduce key definitions. For any tree $T\in\S$, we denote its \emph{set of stems} by $\rS(T)$.
Moreover, we define the \emph{spine height} of each stem  $s\in\rS(T)$, denoted by $h_{\mathrm{spine}}(s)$, as follows: let $v$ denote the unique adjacent vertex of $s$, and let $\varnothing_{{T}}=v_0,\ldots,v_k=v$ be the unique path from the root vertex of ${T}$ to $v$. We set $h_{\mathrm{spine}}(s)$ as the largest index $j$ such that $v_j$ lies on the spine of ${T}$.

\begin{figure}[t!]
  	\centering
  	\includegraphics[width=0.3\linewidth ,page=6]{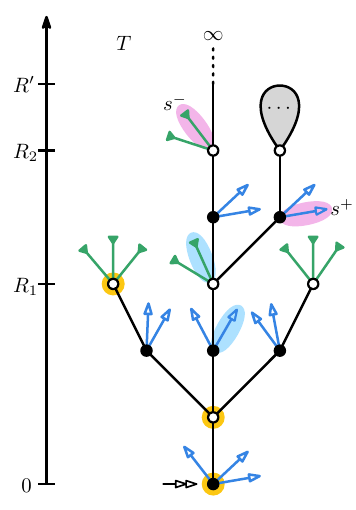}\quad
  	\includegraphics[width=0.65\linewidth ,page=7]{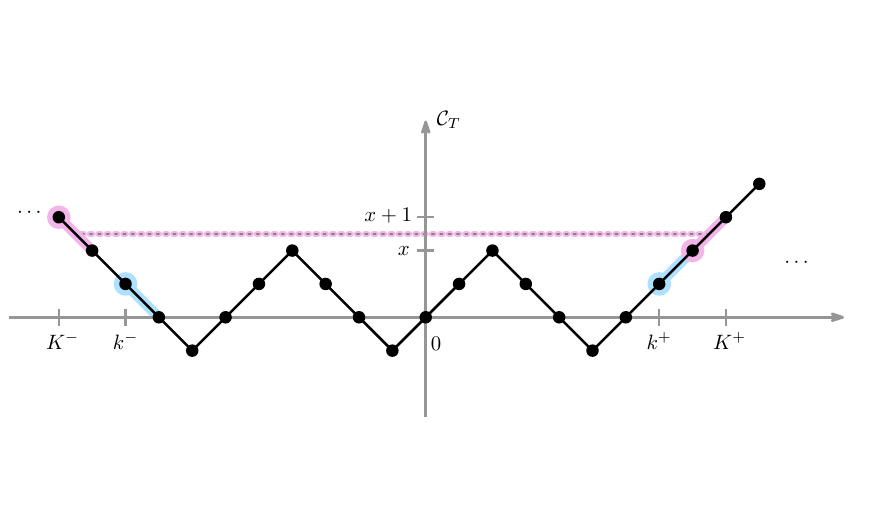}
 	\caption{A tree  $T$ in $\S$ (left) and its associated contour walk of stems $\cC_T$ (right). Some stems of interest are colored in blue and red, in $T$ and $\cC_T$. The vertices of $T$ belonging to $B_1(\Phi(T))$ are highlighted in yellow.}\label{fig:ClosureContinuous}
\end{figure}

\begin{proof}[Proof of Theorem~\ref{theo:ContinuityPhi}]
	Let $T\in\cS^d$ and let $R>0$. Recall the definition of the local distance, see~\eqref{eq:DefLocDist}. We show that there exists $R'>0$ such that for all $T'\in\cS^d$, if $B_{R'}(T)=B_{R'}(T')$, then 
\begin{equation}
	B_{R}(\Phi(T))=B_{R}(\Phi(T')).
\end{equation}
If $T$ is finite, then this statement holds by taking any $R'>0$ such that $B_{R'}(T)=T$. Thus, we may assume that $T$ is infinite.
	
	We adapt the approach of~\cite{Stephenson18}. The subsequent constructions are illustrated in Figure~\ref{fig:ClosureContinuous}. First, choose $R_1>0$ sufficiently large so that all vertices in $B_R(\Phi(T))$ are contained in $B_{R_1}(T)$, i.e.
\begin{equation*}
	\rV\big(B_{R}(\Phi(T))\big)\subset \rV\big(B_{R_1}(T)\big).
\end{equation*}
The existence of $R_1$ follows from the fact that each vertex of $T$ has degree $d$. Next, observe that there are only finitely many stems in $T$ incident to the vertices in $B_{R_1}(T)$, so that we can denote by $k^-$ the minimal index among these stems, and by $k^+\geq k^-$ the maximal index. Set $x = \max_{k^-\leq k\leq k^+}{\cC_T(k)}$ and define 
\begin{align*}
	K^- &= \sup\{k<k^-:~ \cC_T(k)>x\},\\
	K^+ &= \inf\{k>k^+:~ \cC_T(k)\geq x\}.
\end{align*} 	
Both $K^-$ and $K^+$ are finite by Proposition~\ref{pro:ClosureIsWellDefined}. By construction, $K^-$ is the index of a closing stem $s^-$, and $K^+$ is the index of an opening stem $s^+$. 

Now, for any $s\in\rS(T)$ with index $k\in[K^-,K^+]$, we have $K^- \leq \sigma_T(k)\leq K^+$ by construction. Hence, no stem with index in $[K^-,K^+]$ is matched to a stem with index in $\Z\setminus [K^-,K^+]$. Let 
\begin{equation}
	R_2=\max(h_{\mathrm{spine}}(s^-),h_{\mathrm{spine}}(s^+)),
\end{equation}
so that no stem with spine height greater than $R_2+1$ is matched to a stem incident to a vertex in $B_{R_1}(T)$. Finally, take
\begin{equation*}
	R'= \max\{h(s) : s\in\rS(T), h_{\mathrm{spine}}(s)\leq R_2\},
\end{equation*} 
where $h(s)$ stands for the height of stem $s$. This yields the following inclusion, which completes the proof: 
\begin{equation*}
	B_{R}(\Phi(T)) \subset \Phi(B_{R'}(T)).\qedhere
\end{equation*} 
\end{proof}

\subsection{Proof of Theorem~\ref{theo:LocalConvergencedRegularMaps}}

We can now prove the main theorem of this paper. Let $\bm{T_\infty}$  be a random tree distributed according to $\sP^d_\infty$, and let $\bm{M_\infty}\coloneqq\Phi({\bm{T_\infty}})$. By Theorem~\ref{theo:ContinuityPhi},  $\bm{M_\infty}$ almost surely belongs to $\Mc$. Our goal is to prove that its law corresponds to the weak limit of the sequence $\left(\P^d_n\right)_n$. 
\medskip

Let $\bm{T}_n$ be a random tree distributed according to $\sP^d_n$, and let  $(\bm{M}_n,\bm{F}_n)=\Phi(\bm{T}_n)$ denote its associated plane map. First, $(\bm{M}_n,\bm{F}_n)$ is uniformly distributed on $\mathcal{M}_n^{d,\brf}$, since $\Phi$ is a one-to-one mapping on this set by Theorem~\ref{theo:BijMapsTrees}. Secondly, any $d$-regular bipartite maps with $n$ vertices share the same number of faces, see~\eqref{eq:NbPlaneMaps}, and hence $\bm{M}_n$ is  uniformly distributed on $\mathcal{M}_n^{d}$. 

We conclude, by applying Theorems~\ref{theo:WeakConvergencedRegularTrees} and~\ref{theo:ContinuityPhi}, which yield the following equalities and convergence in distribution:
\begin{equation*}
	\bm{M}_n = \Phi(\bm{T}_n) \underset{n\to\infty}{\rightarrow} \Phi(\bm{T_\infty}) = \bm{M_\infty},
\end{equation*}
where, by an abuse of notation, $\Phi(\bm{T}_n)$ means that we apply $\Phi$ to $\bm{T}_n$, after which the marked face in the resulting map is ignored.

\subsection{Proof of Lemma~\ref{pro:ClosureIsWellDefined}}\label{sec:ProofLemma}

\begin{figure}[t!]
  	\centering
  	\includegraphics[width=0.25\linewidth ,page=4]{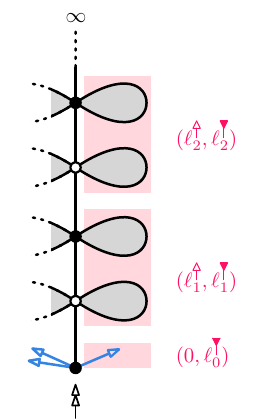}\qquad
  	\includegraphics[width=0.6\linewidth ,page=5]{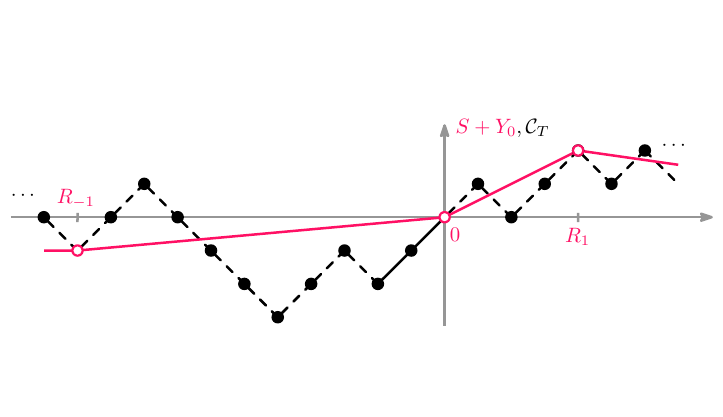}
 	\caption{A tree $T$ in $\S$ (left) and the sub-walk of interest of $\cC_T$  (right). }\label{fig:ReducedWalkRegular}
\end{figure}

Let $\bm{T_\infty}$ be a random tree distributed according to $\sP^d_\infty$. To prove the statement, we construct a suitable subsequence of $(\cC_{\bm{T_\infty}}(k))_k$ which is almost surely recurrent, both for $k\in\Z_{\geq 0}$ and $k\in\Z_{< 0}$. We focus on $k\in\Z_{\geq 0}$; the other case follows by symmetry. The subsequent constructions are illustrated in Figure~\ref{fig:ReducedWalkRegular}.
\medskip

First, as $\bm{T_\infty}$ almost surely has  a single spine, $\rS(\bm{T_\infty})$ can be partitioned into four sets as follows
\begin{equation}
	\rS(\bm{T_\infty}) = \rS^{\cstem}_-(\bm{T_\infty})\sqcup\rS^{\ostem}_-(\bm{T_\infty})\sqcup\rS^{\cstem}_+(\bm{T_\infty})\sqcup\rS^{\ostem}_+(\bm{T_\infty}),
\end{equation}
where $\rS^{\cstem}_-(\bm{T_\infty})$ and  $\rS^{\ostem}_-(\bm{T_\infty})$ consists of the closing and opening stems, respectively, lying on the left of the spine, and $\rS^{\cstem}_+(\bm{T_\infty})$ and  $\rS^{\ostem}_+(\bm{T_\infty})$ consist of the closing and opening stems, respectively, lying on the right of the spine. Then, for every $k\in\Z_{\geq 0}$, define
\begin{align*}
	\ell^{\cstem}_k \coloneqq \mathrm{Card}\{s\in\rS^{\cstem}_+(\bm{T_\infty}): ~2k-1\leq h_{\mathrm{spine}}(s)\leq 2k\},\\
\ell^{\ostem}_k \coloneqq \mathrm{Card}\{s\in\rS^{\ostem}_+(\bm{T_\infty}): ~2k-1\leq h_{\mathrm{spine}}(s)\leq 2k\}.	
\end{align*}
By the branching property of $\sP^d_\infty$, the pairs $(\ell^{\cstem}_k,\ell^{\ostem}_k)$ are independent. For each $k\in\Z_{\geq 0}$, define
\begin{equation*}
	X_k = \ell^{\cstem}_k + \ell^{\ostem}_k\quad\text{and}\quad Y_k = \ell^{\cstem}_k - \ell^{\ostem}_k.
\end{equation*}
Observe first that $X_k$ is almost surely non-negative and has a positive probability of being positive. Hence, the associated walk $R_n=\sum_{k=0}^{n}X_k$ is non-decreasing and diverges to infinity almost surely.

Now, we determine the distribution of $Y_k$. We prove that $Y_0$ is uniformly distributed on $\{0,\ldots,d-1\}$ and that, for all $k\geq 1$, 
\begin{equation}\label{eq:DistributionXkRegular}
	Y_k \overset{(d)}{=} \mathrm{Unif}\{0,\ldots,d-2\}-\mathrm{Unif}\{0,\ldots,d-2\},
\end{equation}
where the two uniform distributions are independent. The following arguments rely on the possible offspring of black and white vertices in $\bm{T_\infty}$. There are described in Proposition~\ref{prop:PropertiesOnInfiniteRandomTree}.

Fix $k\geq 1$. The black vertex at height $2k$ on the spine has a unique white child, which lies on the spine, and $d-2$ incident opening stems. The ordering is chosen uniformly among the $d-1$ configurations. Thus, the number of opening stems to the right of the spine with spine height $2k$ is uniformly distributed on $\{0,\ldots,d-2\}$.  This yields the term $-\mathrm{Unif}\{0,\ldots,d-2\}$ in~\eqref{eq:DistributionXkRegular}.
\medskip

Consider now the white vertex $v$ at height $2k-1$. It has one black child, which lies on the spine, together with a collection of $d-2$ elements which are either closing stems or black vertices of charge $1$. The ordering is again  uniform. Thus, the number of stems and black children of $v$ lying to the right of the spine is uniformly distributed on $\{0,\ldots,d-2\}$. Denote this number by $p$, so that their total contribution to $Y_k$ is $p$, since each of them contributes by $1$ to $Y_k$. This yields the term $+\mathrm{Unif}\{0,\ldots,d-2\}$ in~\eqref{eq:DistributionXkRegular}.
\medskip

We conclude that $S_n=\sum_{k=1}^{n}Y_k$ is a symmetric random walk on $\Z$ with no drift and independent steps. Thus, it is recurrent. Finally, for all $n\geq 1$,
\begin{equation*}
	\cC_{\bm{T_\infty}}(R_n) = S_n+Y_0. 
\end{equation*}
This identity follows directly from grouping together the stems lying to the right of the spine according to their spine height. This completes the proof.

\section{Perspectives for bipartite maps with general degree distribution and the Ising model}\label{sec4}

In this section, we outline several ongoing research directions concerning the local limits of random bipartite maps with general degree distributions and random maps endowed with an Ising configuration.

\paragraph*{Local limit of bipartite maps with general degree distribution}
First, we expect that Theorem~\ref{theo:LocalConvergencedRegularMaps} to also hold for bipartite planar maps with bounded degrees and with a prescribed degree distribution. More precisely, let $\mathcal{M}_n$ be the set of all bipartite planar maps with $n$ vertices and let $q_\bullet,q_\circ:\Z_{\geq 0}\rightarrow\R^+$ be two weight functions. Let $\bm{M}_n$ be a random bipartite map sampled from $\mathcal{M}_n$ with probability proportional to its weight, defined by
\begin{equation*}
	\prod_{v\in\rV_{\!\bullet}(\rm)}q_\bullet({\deg(v)})\prod_{v\in\rV_{\!\circ}(\rm)}q_\circ({\deg(v)}).
\end{equation*}

In the case of bounded degrees, that is, when $q_\bullet$ and $q_\circ$ are nonzero only for a finite set of indices, we expect $\mathcal{M}_n$ to converge in distribution, for the local topology, to a random one-ended bipartite planar map whose distribution depends on $q_\bullet$ and $q_\circ$. 

The proof in this setting should be similar in spirit to the one provided here. Indeed, it would rely on the same blossoming bijection and the argument would proceed in the same steps. First, we establish the  local convergence of  the corresponding family of blossoming trees. Second, we extend the bijection to infinite trees, allowing the local convergence to be transferred to maps. However, the argument would be more elaborate than in the present work, since there are no explicit enumerative formulas in this broader setting. Instead, we could rely on the Drmota-Lalley-Wood Theorem to study the generating functions of blossoming trees analytically, together with the general results of Stephenson~\cite{Stephenson18} regarding the local limit of multitype Bienaymé-Galton-Watson trees.  

A natural further step would be to extend this result to the case of unbounded degrees. However, this would require additional theoretical results, as the family of trees arising from the bijection is more complex. Nevertheless, we expect a similar local convergence and a corresponding limiting distribution.

\paragraph*{Local limit of maps endowed with Ising configuration }
As a continuation of the previous perspective, we also expect to derive the local convergence of planar maps decorated with an Ising configuration from the local convergence of bipartite plane maps with vertex degree $2$ and $d$. Let us be more precise. 
A \emph{spin configuration} on a map $\rm$ is a mapping $\sigma:\rV(\rm)\rightarrow\{-1,+1\}$. An edge $e\in\rE(\rm)$ is said to be \emph{monochromatic} if its two endpoints are mapped to the same value. The Ising model on planar maps consists in studying random finite planar maps endowed with a spin configuration, where the probability of sampling a map $(\rm,\sigma)$ is proportional to $\nu^{m(\rm,\sigma)}$, where $m(\rm,\sigma)$ denotes the number of monochromatic edges of $\rm$, and where $\nu$ is a fixed non-negative real number (representing the inverse temperature of the model). 
\medskip

A well-known method for studying maps endowed with a spin configuration relies on a connection with bipartite maps. Consider a map endowed with a spin configuration $(\rm,\sigma)$, and for each vertex $v\in\rV(\rm)$, color $v$ in black if $\sigma(v)=-1$ and in white if $\sigma(v)=+1$. Note that, at this stage, $\rm$ may not be bipartite in terms of this vertex-coloring. Then replace each edge of $\rm$ by a chain of arbitrary length consisting of black and white vertices of degree $2$, so that the resulting map becomes bipartite. Necessarily, a monochromatic edge is replaced by a chain of odd length, whereas a non-monochromatic is replaced by a chain of even length. This transformation allows to get a functional equation between the generating functions of both models, see~\cite{AMT25,BousquetMelouSchaeffer03}. Observe that if $\rm$ is $d$-regular, then the bipartite maps obtained in this manner have vertex degree exactly $2$ and $d$. 
\medskip

We expect that this transformation could lead to a proof of local convergence of $d$-regular maps endowed with an Ising configuration, at least in the regime $\nu\leq 1$, which corresponds to the anti-ferromagnetic Ising model. 

\newpage
\bibliography{lipics-v2021-sample-article}

\newpage
\appendix
\section{Multitype Bienaymé-Galton-Watson and proof of Proposition~\ref{prop:PropertiesOnInfiniteRandomTree}}\label{appA}

The purpose of this appendix is to prove Proposition~\ref{prop:PropertiesOnInfiniteRandomTree}, whose proof relies on the theory of multitype Bienaymé-Galton-Watson trees. 
We begin by recalling useful definitions concerning multitype trees and random multitype Bienaymé-Galton–Watson trees; see~\cite[Section 2.1]{Stephenson18} for definitions with the Ulam formalism.

\subsection{Definitions and notation}\label{sec:DefinitionsMultitype}

\paragraph*{Multitype trees}
Let $\cK$ be a finite set. A \emph{$\cK$-type tree} is a pair $(\rt,\kappa)$ consisting of a plane tree  $\rt$ (possibly infinite), either rooted or planted, together with a function $\kappa:\rV(\rt)\rightarrow \cK$. The value $\kappa(v)$ is called the \emph{type} of the vertex $v$. When the context is clear, we simply write $\rt$ instead of $(\rt,\kappa)$. The root of such a tree $\rt$ is denoted by $\varnothing_\rt$. We denote by $\cT^{\cK,(i)}$ the set of all finite $\cK$-type trees $\rt$ for which $\varnothing_\rt$ has type $i\in\cK$.

For any $\cK$-type tree $\rt$ and $v\in\rV(\rt)$, the \emph{list of types of the ordered offspring} of $v$ in $\rt$ is defined by
\begin{equation*}
\sw_\rt(v)\coloneqq\left(\kappa(u_1),\ldots,\kappa(u_l)\right),
\end{equation*}
where $\left(u_1,\ldots,u_l\right)$ denotes the ordered offspring of $v$. 

Finally, we denote by $\sW_\cK\coloneqq\cup_{n\geq 0}{\cK^n}$ the set of all finite lists of types. For $\sw\in\sW_\cK$, denote its size by $\vert\sw\vert$ and for any $i\in\cK$ set:
\begin{equation*}
	\vert\sw\vert_i \coloneqq \big\vert\left\{1\leq j\leq \vert\sw\vert  : \sw_j=i\right\}\big\vert.
\end{equation*}

\paragraph*{Multitype Bienaymé-Galton–Watson distribution}
First, an \emph{ordered offspring distribution} is a sequence $\bm{\zeta}=\left(\zeta^{(i)}\right)_{i\in\cK}$, where ${\zeta}^{(i)}$ is a probability measure on $\sW_\cK$.  We then define the distribution of a \emph{$\cK$-type Galton–Watson tree} rooted at a vertex of type $i\in\cK$ and with ordered offspring distribution $\bm{\zeta}$ as the following measure:
\begin{equation}
	\mathbb{P}_{\bm{\zeta}}^{(i)}\!\left(\rt,\kappa\right) \coloneqq \1_{\{\kappa(\varnothing_\rt)=i\}}\prod\limits_{v\in\rV(\rt)}{\zeta^{(\kappa(v))}{\big(\sw_\rt(v)\big)}},
\end{equation}
for any finite tree $(\rt,\kappa)\in\cT^{\cK,(i)}$. 
Note that $\mathbb{P}_{\bm{\zeta}}^{(i)}$ is, in general, only a sub-probability. It becomes a probability measure in the critical case, which is defined below. In that case, a tree sampled from this distribution is referred to as a \emph{$\bm{\zeta}$-Bienaymé-Galton-Watson tree with root type $i$}.
 
\paragraph*{Criticality, Extinction, and Irreducibility}\label{sec:DefIrreducibilityCriticality}
For ease of notation, we index $\cK=\{\kappa_1,\ldots,\kappa_K\}$. We define the \emph{mean matrix} of $\bm{\zeta}$ as the matrix $M=\left(m_{i,j}\right)_{1\leq i,j \leq K}$, where $m_{i,j}$ is the expected number of type $\kappa_j$ children of a type $\kappa_i$ vertex under $\zeta^{(\kappa_i)}$:
\begin{equation*}
	m_{i,j} \coloneqq \mathbb{E}\!\left(\vert\sw^{(i)}\vert_j\right) = \sum\limits_{\sw\in\sW_\cK}{\vert\sw\vert_{\kappa_j}\,\zeta^{(\kappa_i)}(\sw)},
\end{equation*}
where $\sw^{(i)}$ is a random variable on $\sW_\cK$ with distribution $\zeta^{(\kappa_i)}$. 

We say that $\bm{\zeta}$, or equivalently $M$, is \emph{critical} if the spectral radius of $M$ is $1$. 

We turn to describe a well-known sufficient condition on  $\bm{\zeta}$ that guarantees that a $\bm{\zeta}$-Bienaymé-Galton-Watson tree is almost surely finite. Specifically, we say that $\bm{\zeta}$ is \emph{non-degenerate} if there exists a type that can produce at least two children, i.e., if there exists $i\in\cK$ such that $\zeta^{(i)}(\{\sw: \vert\sw\vert>1\})>0$.

\begin{proposition}[\cite{Miermont08}] If $\bm{\zeta}$ is non-degenerate and critical, then a $\bm{\zeta}$-Bienaymé-Galton-Watson tree, regardless of the root type, is almost surely finite.
\end{proposition}

Finally, the offspring distribution $\bm{\zeta}$, or equivalently $M$, is said to be \emph{irreducible} when for all $i,j\in\{1,\ldots,K\}$, there exists $p\in\Z_{>0}$ such that the $(i,j)$-th entry of $M^p$ is non-zero. 
\medskip

\begin{remark}
Usually, $\bm{\zeta}$ is required to be irreducible and critical. In that case, the Perron-Frobenius theorem applies and ensures that: 
\begin{itemize}
	\item The integer $1$ is an eigenvalue of $M$ with multiplicity one,
	\item The unique, up to multiplicity, left and right eigenvectors of $M$ associated to the eigenvalue $1$ have non-negative coefficients. More precisely, we denote them by $\bm{a}$ and $\bm{b}$, normalized such that $\sum_i{a_i}=1$ and $\sum_i{a_ib_i}=1$.
	\item Moreover, all the coefficients of $\bm{a}$ and $\bm{b}$ are positive.
\end{itemize}

In next section, $\bm{\zeta}$ will be critical but not irreducible. Still, it will satisfy the first two consequences of the Perron-Frobenius theorem. In this case, we then say that $\bm{\zeta}$ is \emph{almost-irreducible-critical}. 

For simplicity, we now write $a_k$ and $b_k$ in place of $a_{\kappa_i}$ and $b_{\kappa_i}$ whenever $k=\kappa_i$ holds.
\end{remark}

\paragraph*{Multitype Bienaymé-Galton-Watson tree conditioned to survive}
Under certain hypotheses (including irreducibility and criticality),
the distribution of a $\bm{\zeta}$-Bienaymé-Galton-Watson tree converges weakly, for the local topology, to the distribution of a multitype Bienaymé-Galton-Watson tree (with more types), called \emph{a $\bm{\zeta}$-Bienaymé-Galton-Watson tree conditioned to survive}, see~\cite{Stephenson18}. Its distribution is denoted by $\widehat{\P}_{\bm{\zeta}}^{(i)}$ when the root is of type $i$.

When $\bm{\zeta}$ is non-degenerate and almost-irreducible-critical $\widehat{\P}_{\bm{\zeta}}^{(i)}$ exists, see~\cite{Stephenson18}. Precisely, a tree $\bm{T}$ sampled with this distribution is an infinite multitype tree that can be described as follows: it consists of an infinite line of descent starting from the root, called the spine, along which independent critical trees with offspring distribution $\bm{\zeta}$ are grafted. On the spine, the vertices follow a \emph{sized-biased} version of $\bm{\zeta}$, denoted by $\bm{\widehat{\zeta}}$, and defined by:
\begin{equation}\label{eq:SizebiasedLaw}
	{\widehat{\zeta}}^{(i)}(\sw)\coloneqq \Big(\frac{1}{b_{i}}\sum\limits_{j=1}^{\vert\sw\vert}b_{\sw_j}\Big)\,\zeta^{(i)}(\sw).
\end{equation}
Then, for each vertex on the spine with offspring $\sw\in\cK$, one of its children is chosen to continue the spine: the index $j$ is  chosen with probability proportional to $b_{\sw_j}$, for $1\leq j\leq \vert\sw\vert$.

\subsection{Proof of Proposition~\ref{prop:PropertiesOnInfiniteRandomTree}}
We proceed as follows. First, we describe an explicit offspring distribution $\bm{\zeta}$. Then we verify its criticality properties. Finally, we identify the $\bm{\zeta}$-multitype Galton-Watson tree conditioned to survive, and we conclude by computing the probability of observing a given ball around its root vertex.

\paragraph*{Definition of the offspring distribution}
First, we define the set of types: let $\cK\coloneqq\{\varnothing,\bullet,\circ,\s\}$. These types are interpreted in the following way: 
\begin{itemize}
	\item $\varnothing$ represents the root vertex, 
	\item $\bullet$ and $\circ$ represent \emph{classical} non-root black vertices and white vertices, respectively,
	\item $\s$ represents the opening and closing stems.
\end{itemize}

\begin{remark}To simplify further explanation, we identify the types and their representations. Conversely, we identify any tree in $\Tc$ with its natural representation as a $\cK$-type tree.
\end{remark}
\medskip

Next, we define the offspring distributions. Roughly speaking,
\begin{itemize}
	\item stems, i.e., vertices with type $\s$, are infertile,
	\item the offspring of a vertex with type $\bullet$ or $\varnothing$ is sampled uniformly from the set of possible offspring described in Lemma~\ref{Cla:ChargeOffspringRegular},
	\item the offspring of a vertex with type $\circ$ can be any offspring described in Lemma~\ref{Cla:ChargeOffspringRegular}, and it is sampled with probability proportional to $\rB(z)^k$ if it consists of $k$ black vertices together with $d-1-k$ stems.
\end{itemize}
This is written formally as follows, where for any an ordered sequence of types $\sw\in\sW_{\{\varnothing,\bullet,\circ,\s\}}$, we denote  
\begin{equation*}
L(\sw)\coloneqq(\vert\sw\vert_\varnothing,\vert\sw\vert_\bullet,\vert\sw\vert_\circ,\vert\sw\vert_\s),
\end{equation*} 
its associated vector that counts the occurrence of each type:
\begin{align}\label{eq:DefOsspringDistributionRegular}
	&\zeta^{(\varnothing)}(\sw) \coloneqq \1_{\{L(\sw)=(0,0,1,d-1)\}} \frac{1}{d},\notag \\ \notag
	&\zeta^{(\bullet)}(\sw) \coloneqq \1_{\{L(\sw)=(0,0,1,d-2)\}} \frac{1}{d-1},\\ 
	&\zeta^{(\circ)}(\sw) \coloneqq \sum_{k=0}^{d-1}\1_{\{L(\sw)=(0,k,0,d-1-k)\}} \frac{(d-2)^{d-1-k}}{(d-1)^{d-1}},\\ 	\notag
	&\zeta^{(\s)}(\sw) \coloneqq \1_{\{\sw=\varnothing\}}. \notag
\end{align}
The formula obtained for vertices with types $\circ$ follows from explicit computations by applying the Lagrange inversion formula on~\eqref{LagrangianEqB1}.

\begin{remark}
As $\s$ is an infertile type, it does not affect criticality but prevents irreducibility of $\bm{\zeta}$.
\end{remark}

Since black vertices produce only white children and white vertices produce only black children, $\bm{\zeta}$-multitype Galton-Watson tree is bipartite. Moreover, by construction, it is almost surely well-charged, in the sense defined in~\eqref{eq:ChargedConditions}. Note that here a stem incident to a black vertex naturally represents an opening stem and a stem incident a white vertex naturally represents a closing stem.

\paragraph*{Properties of the offspring distribution}
We prove that $\bm{\zeta}$ is non-degenerate and almost-irreducible-critical, so that $\widehat{\P}_{\bm{\zeta}}^{(i)}$ exists. First, non-degeneracy is immediate by definition. Then, by direct computations, the mean matrix of $\bm{\zeta}$ is equal to:
\begin{equation*}
	\begin{pNiceArray}{cccc}
  		0 & 0 & 1 & d-1 \\
  		0 & 0 & 1 & d-2 \\
  		0 & 1 & 0 & d-2 \\
  		0 & 0 & 0 & 0  
	\end{pNiceArray},
\end{equation*} 
where the order of the types is $\varnothing$, $\bullet$, $\circ$ and then $\s$. 
Clearly, $1$ is an eigenvalue of multiplicity one, and therefore  $\bm{b}_\varnothing=\bm{b}_\bullet=\bm{b}_\circ=d-1$ and $\bm{b}_\s=0$. Thus, $\bm{\zeta}$ is almost-irreducible-critical.

We finish by computing the sized-biased version of $\bm{\zeta}$:
\begin{equation*}
\widehat{\zeta}^{(i)}(\sw)=
	\begin{cases}
			{\zeta}^{(i)}(\sw) & \text{for } i=\s,\\
			\vert\sw\vert_\circ\,{\zeta}^{(i)}(\sw)={\zeta}^{(i)}(\sw) & \text{for } i\in\{\bullet,\varnothing\},\\
			\vert\sw\vert_\bullet\,{\zeta}^{(i)}(\sw) & \text{for } i=\circ,
	\end{cases}
\end{equation*}
since root and black vertices have exactly one white child, we have $\vert\sw\vert_\circ=1$ almost surely under $\zeta^{(\varnothing)}$ and $\zeta^{(\bullet)}$.

\paragraph*{End of the proof}
Let $\bm{T}$ be $\bm{\zeta}$-multitype Galton-Watson tree conditioned to survive. First, the previous paragraphs yield that $\bm{T}$ verifies the structural claims of Proposition~\ref{prop:PropertiesOnInfiniteRandomTree}. We conclude by computing the following probability:
\begin{equation}
	\widehat{\P}_{\bm{\zeta}}^{(0)}\big(B_k(\bm{T})= B_k(\rt)\big),
\end{equation}
where $k\in\Z_{\geq 0}$ and $\rt\in\T$ are fixed.
Following Section~\ref{subsec:ProofLocalConvTree}, we denote by $m_k(\rt)$ the number of vertices of $\rt$ at height $k$, and we set $n_k(\rt)=\vert\rV_{\!\bullet}(B_k(\rt))\vert$. First, the spine intersects level $k$ at exactly one vertex, and conditionally on the truncated tree, all $m_k(\rt)$ vertices at height $k$ are equally likely to be on the spine. Thus,
\begin{equation}
	\widehat{\P}_{\bm{\zeta}}^{(0)}\big(B_k(\bm{T})= B_k(\rt)\big) = m_k(\rt)\cdot{\P}_{\bm{\zeta}}^{(0)}\big( B_k(\bm{T}) = B_k(\rt)\big),
\end{equation}
This quantity depends on the parity of $k$ since $\bm{T}$ is almost surely bipartite. For odd values of $k$, we then get
\begin{align*}
	\widehat{\P}_{\bm{\zeta}}^{(0)}\big(& B_k(\bm{T})= B_k(\rt)\big)   \\
		&= \frac{m_k(\rt)}{d}\left(\frac{1}{d-1}\right)^{n_k(\rt)-1}\left(\frac{(d-2)^{d-1}}{(d-1)^{d-1}}\right)^{\vert\rV_{\!\circ}(B_{k-1}(\rt))\vert}(d-2)^{\sum_{v\in\rV_{\!\circ}(B_{k-1}(\rt))}{\vert\sw(v)\vert_{\bullet}}}\\
		&= m_k(\rt)\,{\rho_d}^{n_k(\rt)}\,\frac{(d-1)(d-2)}{d}  \left(\frac{d-1}{d-2}\right)^{(d-1)\, m_k(\rt)},
\end{align*}
since by counting,  $n_k(\rt)=m_k(\rt)+\vert\rV_{\!\circ}(B_{k-1}(\rt))\vert$ and $\sum_{v\in\rV_{\!\circ}(B_{k-1}(\rt))}{\vert\sw(v)\vert_{\bullet}}=n_k(\rt)-1$. Thus, we get that 
\begin{equation*}
	\P\big(B_k(\bm{T}_n) = B_k(\rt)\big) \underset{n\to\infty}{\longrightarrow} \widehat{\P}_{\bm{\zeta}}^{(0)}\big( B_k(\bm{T}) = B_k(\rt)\big).
\end{equation*}
The case where $k$ is even is similar and is left to the reader.
This completes the proof of Proposition~\ref{prop:PropertiesOnInfiniteRandomTree}.

\begin{remark}
Note that if $\bm{\zeta}$ was irreducible and critical, we could have provided a more direct proof of Theorem~\ref{theo:WeakConvergencedRegularTrees} and  Proposition~\ref{prop:PropertiesOnInfiniteRandomTree}. Indeed, it is easily seen that a $\bm{\zeta}$-Bienaymé-Galton-Watson tree conditioned to have $n$ black vertices is uniformly distributed on $\cT_n$, up to the identification introduced above. Then the general result of~\cite[Theorem 3.1]{Stephenson18} would imply  their convergence in distribution, for the local topology, to the law of the $\bm{\zeta}$-Bienaymé-Galton-Watson tree conditioned to survive. 

However, $\bm{\zeta}$ is not irreducible, and our proof therefore relies on explicit enumerative formulas for maps and trees in the case of $d$-regular maps.

Alternatively we could have reduced the model to an irreducible one, by removing the stems, then apply the Stephenson theorem and finally glue back the stems. We expect this strategy to work in the context of arbitrary vertex degrees.
\end{remark}

\end{document}